\documentclass[12pt]{amsart} 

\usepackage[OT2, T1]{fontenc}

\usepackage{amscd}
\usepackage{amsmath}
\usepackage{amssymb}
\usepackage{mathrsfs} 			
\usepackage{units}
\usepackage[all]{xy}

\usepackage{algorithm}
\usepackage{algorithmic}

\def\frk{\frak}               

\def\mm{{\frk m}}

\def\Phi{{\frk n}}
\def\Phi{{\frk N}}
\def\opn#1#2{\def#1{\operatorname{#2}}} 

\opn\projdim{proj\,dim} \opn\injdim{inj\,dim} \opn\rank{rank}
\opn\depth{depth} \opn\sdepth{sdepth} \opn\fdepth{fdepth}
\opn\grade{grade} \opn\height{height} \opn\embdim{emb\,dim}
\opn\codim{codim}  \opn\min{min} \opn\max{max}

\opn\Tr{Tr} \opn\bigrank{big\,rank}
\opn\superheight{superheight}\opn\lcm{lcm}
\opn\trdeg{tr\,deg}
\opn\reg{reg} \opn\lreg{lreg} \opn\ini{in} \opn\lpd{lpd}
\opn\size{size}
\opn\div{div} \opn\Div{Div} \opn\cl{cl} \opn\Cl{Cl}
\opn\Spec{Spec} \opn\Supp{Supp} \opn\supp{supp} \opn\Sing{Sing}
\opn\Ass{Ass} \opn\Min{Min}
\opn\Ann{Ann} \opn\Rad{Rad} \opn\Soc{Soc}
\opn\Im{Im} \opn\Ker{Ker} \opn\Coker{Coker} \opn\Am{Am}
\opn\Hom{Hom} \opn\Tor{Tor} \opn\Ext{Ext} \opn\End{End}
\opn\Aut{Aut} \opn\id{id}  \opn\deg{deg}

\opn\nat{nat}
\opn\pff{pf}
\opn\Pf{Pf} \opn\GL{GL} \opn\SL{SL} \opn\mod{mod} \opn\ord{ord}
\opn\Gin{Gin} \opn\Hilb{Hilb}
\opn\aff{aff} \opn\con{conv} \opn\relint{relint} \opn\st{st}
\opn\lk{lk} \opn\cn{cn} \opn\core{core} \opn\vol{vol}
\opn\link{link} \opn\star{star}
\opn\gr{gr}

\def\pot#1#2{#1[\kern-0.28ex[#2]\kern-0.28ex]}

\opn\dirlim{\underrightarrow{\lim}}
\opn\inivlim{\underleftarrow{\lim}}
\def\Implies{\ifmmode\Longrightarrow \else
        \unskip${}\Longrightarrow{}$\ignorespaces\fi}
\def\implies{\ifmmode\Rightarrow \else
        \unskip${}\Rightarrow{}$\ignorespaces\fi}
\def\iff{\ifmmode\Longleftrightarrow \else
        \unskip${}\Longleftrightarrow{}$\ignorespaces\fi}

\let\:=\colon
\newtheorem{Theorem}{Theorem}[]
\newtheorem{Lemma}[Theorem]{Lemma}

\newtheorem{Proposition}[Theorem]{Proposition}

\theoremstyle{definition}

\newtheorem{Remark}[Theorem]{Remark}

\newtheoremstyle{subsection-tweak}
   {11pt}
   {3pt}%
   {}
   {}%
   {\bfseries}
   {}%
   {.5em}
   {\thmnumber{\@{#1}{}\@{#2}.}%
    \thmnote{~{\bfseries#3.}}}    

\newcounter{numberingbase}

\theoremstyle{subsection-tweak}
\newtheorem{bpp}[Theorem]{}
\newtheorem{bppt}[numberingbase]{}
\newcommand{\bbpp}{\begin{bpp}}
\newcommand{\eepp}{\end{bpp}}
\newcommand{\bbppt}{\begin{bppt}}
\newcommand{\eeppt}{\end{bppt}}

\theoremstyle{theorem}

\theoremstyle{definition}

\newcommand{\val}{\mathrm{val}}		

\providecommand{\qxq}[1]{\quad\text{#1}\quad}

\newcommand{\tst}{\textstyle}

\DeclareMathOperator{\Frac}{Frac}		

\let\epsilon\varepsilon
\let\phi=\varphi
\def\qed{\ifhmode\textqed\fi
      \ifmmode\ifinner\quad\qedsymbol\else\dispqed\fi\fi}
\def\textqed{\unskip\nobreak\penalty50
       \hskip2em\hbox{}\nobreak\hfil\qedsymbol
       \parfillskip=0pt \finalhyphendemerits=0}
\def\dispqed{\rlap{\qquad\qedsymbol}}

\opn\dis{dis}
\def\pnt{{\raise0.5mm\hbox{\large\bf.}}}

\opn\Lex{Lex}

\begin{document}

\title{Algebraic valuation ring extensions as limits of complete intersection algebras}

\author{ Dorin Popescu}

\address{Simion Stoilow Institute of Mathematics of the Romanian Academy,
Research unit 5, P.O. Box 1-764, Bucharest 014700, Romania,}

\address{University of Bucharest, Faculty of Mathematics and Computer Science
Str. Academiei 14, Bucharest 1, RO-010014, Romania,}

\address{ Email: {\sf dorin.m.popescu@gmail.com}}

\begin{abstract} We show that an algebraic immediate valuation ring extension   of characteristic $p>0$ is a filtered union of   complete intersection algebras of finite type. 

\end{abstract}

\maketitle

\section*{Introduction}

  B. Antieau and  R. Datta have recently proven a positive characteristic analogue  \cite[Theorem 4.1.1]{AD}  of Zariski's theorem \cite{Z}. It says that every perfect valuation ring of characteristic $p>0$ is a filtered union of its smooth ${\bf F}_p$-subalgebras. This result is an application of \cite[Theorem 1.2.5]{T} which relies on some results from \cite{J}. Also E. Elmanto and M. Hoyois proved that an absolute integrally closed valuation ring of residue field of characteristic $p>0$ is a filtered union of its regular finitely generated  ${\bf Z}$-subalgebras (see \cite[Corollary 4.2.4]{AD}). We remind that a  filtered direct limit (in other words a filtered colimit) is a limit indexed by a small category that is filtered (see \cite[002V]{SP} or \cite[04AX]{SP}). A filtered  union is a filtered direct limit in which all objects are subobjects of the final colimit, so that in particular all the transition arrows are monomorphisms.
 
  It is well known that if the fraction field extension of an immediate extension $V\subset V'$ is finite and $p>0$ then $V'$ may fail to be  a filtered direct limit of  smooth $V$-algebras as shows  \cite[Example 3.13]{Po1} inspired from \cite{O} (see also \cite[Remark 6.10]{Po1}). An inclusion $V \subset V'$ of valuation rings is an \emph{immediate extension} if it is local as a map of local rings and induces isomorphisms between the value groups and the residue fields of $V$ and $V'$.  
 After seeing \cite[Theorem 6.2]{KT} (see also \cite{T1}, \cite{T2})  we understood that in general we should expect that $V'$ is a  filtered  union of its complete intersection $V$-subalgebras. In the Noetherian case a morphism of rings is a filtered direct limit of smooth algebras iff it is a regular morphism (see \cite{Po0}, \cite{S}).

A {\em complete intersection} $V$-algebra essentially of finite type is a local $V$-algebra of type $C/(P)$, where $C$ is a localization of a polynomial $V$-algebra of finite type and $P$ is a regular sequence of elements of $C$.  Theorem \ref{T} stated below says that $V'$ is a  filtered
 union of its $V$-subalgebras of type $C/(P)$. Since $V'$ is local it is enough to say that $V'$ is  a  filtered
 union of its $V$-subalgebras of type $T_h/(P)$,  $T$ being a polynomial $V$-algebra of finite type, $0\not =h\in T$ and $P$ is a regular sequence of elements of $T$. Clearly, $ T_h$ is a smooth $V$-algebra and in fact it is enough to say that $V'$ is  a  filtered
 union of its $V$-subalgebras of type $G/(P)$, where $G$  is a smooth $V$-algebra of finite type and $P$ is a regular sequence of elements of $G$. Conversely, a $V$-algebra of such type $G/(P)$ has the form $T_h/(P)$ for some $T,h,P$ using \cite[Theorem 2.5]{S}. By abuse we understand by a {\em complete intersection } $V$-algebra of finite type a $V$-algebra of such type $G/(P)$, or $T_h/(P)$ which are not assumed to be flat over $V$.

The goal
of this paper is to establish the following theorem.

\begin{Theorem}\label{T} Let  $ V'$ be an  immediate extension  of a valuation ring $V$  and  $K\subset K'$ the fraction field extension. If $K'/K$ is algebraic  then $V'$ is a filtered
 union of its complete intersection $V$-subalgebras of finite type.
\end{Theorem}
The proof relies on Proposition \ref{pr} which uses hard results from \cite{P1}. Lemma \ref{1} is the first step in the proof of Proposition \ref{pr}.

We owe thanks to Arnab Kundu who hinted us a gap in the proof of a former version of this paper.

\vskip 0.5 cm

\section{Algebraic immediate extensions of valuation rings}

An inclusion $V \subset V'$ of valuation rings is an \emph{immediate extension} if it is local as a map of local rings and induces isomorphisms between the value groups and the residue fields of $V$ and $V'$.

Let $\lambda$ be a fixed limit ordinal  and $v=\{v_i \}_{i < \lambda}$ a sequence of elements in $V$ indexed by the ordinals $i$ less than  $\lambda$. Then $v$ is \emph{pseudo convergent} if 

$\val(v_{i} - v_{i''} ) < \val(v_{i'} - v_{i''} )     \ \ \mbox{for} \ \ i < i' < i'' < \lambda$
(see \cite{Kap}, \cite{Sch}).
A  \emph{pseudo limit} of $v$  is an element $w \in V$ with 

$ \val(w - v_{i}) < \val(w - v_{i'}) \ \ \mbox{(that is,} \ \ \val(w -  v_{i}) = \val(v_{i} - v_{i'})) \ \ \mbox{for} \ \ i < i' < \lambda$. We say that $v$  is 
\begin{enumerate}
\item
\emph{algebraic} if some $f \in V[T]$ satisfies $\val(f(v_{i})) < \val(f(v_{i'}))$ for large enough $ i < i' < \lambda$;

\item
\emph{transcendental} if each $f \in V[T]$ satisfies $\val(f(v_{i})) = \val(f(v_{i'}))$ for large enough $i < i' < \lambda$.
\end{enumerate}

 The following  lemma is a variant  of  Ostrowski (\cite[ page 371, IV and III]{O}, see also \cite[(II,4), Lemma 8]{Sch}).

\begin{Lemma}(Ostrowski) \label{o1} Let $\beta_1,\ldots,\beta_m$ be any elements of an ordered abelian group $G$, $\lambda$ a limit ordinal and
 let $\{\gamma_s\}_{s<\lambda}$ be a well-ordered, monotone increasing set of elements of G, without
a last element. Let $ t_1,\ldots, t_m$, be distinct integers. Then there  exists
an ordinal $\nu<\lambda$ such that $\beta_i+t_i\gamma_s$ are different for all $s>\nu$. In particular,  
there exists an integer $1\leq r\leq m$ such that
$$\beta_i+t_i\gamma_s>\beta_r+t_r\gamma_s$$ 
for all $i\not = r$ and $s>\nu$.
\end{Lemma}
\begin{proof} The known proof from  the quoted papers gives the second statement even when the integers  $t_i$ are not necessarily positive. Then apply  this statement iteratively to get the first statement.
\hfill\ \end{proof}

\begin{Lemma} \label{ka}
Let $V \subset V'$ be an  immediate extension of valuation rings, $K\subset K'$ its fraction field extension and  $(v_i)_{i<\lambda}$ an algebraic pseudo convergent sequence in $V$, which has a pseudo limit $x$ in $V'$, but no pseudo limit in $K$. Set\\
 $x_i=(x-v_i)/(v_{i+1}-v_i)$. Let $s\in {\bf N}$ be   the
 minimal degree of the polynomials $f\in V[Y]$ such that
 $\val(f(v_i))<\val(f(v_j))$ for large $i<j<\lambda$ and $g\in V[Y]$ a polynomial with $\deg g<s$. Then there exist $d\in V\setminus \{0\}$ and $u\in V[x_i]$ for some $i<\lambda$ with $g(x)=du$ and $\val(u)=0$.
\end{Lemma}

\begin{proof}  In the Taylor expansion\footnote{The polynomials $D^{(n)}f \in R[Y]$ for $f \in R[Y]$, the so called Hasse-Schmidt derivatives, make sense for any ring $R$: indeed, one constructs the Taylor expansion in the universal case $R = {\bf Z}[a_0, \dots, a_{\deg f}]$ by using the equality $n! \cdot (D^{(n)} f) = f^{(n)}$ and verifying over $\Frac(R)$.}

\[
\tst  g(x) = \sum_{n = 0}^{\deg g} (D^{(n)}g)(v_i) \cdot (x - v_i)^n \qxq{with}  D^{(n)}g \in V[Y]
\]	
the values $\val((D^{(n)}g)(v_i) \cdot (x - v_i)^n)$ are pairwise distinct for all $n\leq \deg g$ and every large enough $i$ because $\val(x-v_i)=\val(v_{i+1}-v_i)$ increases,  $\deg (D^{(n)}g)< s$ and so $\val((D^{(n)}g)(v_i))$ is constant for all $n\leq \deg g$ and $i$ large (we could also apply Lemma \ref{o1}). 
We cannot have 
 $\val((D^{(n)}g)(v_i) \cdot (x - v_i)^n)<\val(g(v_i))$
for some $n>0$ and $i$ large because otherwise we get from above that
$\val(g(x))= \val((D^{(n)}g)(v_i) \cdot (x - v_i)^n)$
increases, which is false.
It follows that
 $\val((D^{(n)}g)(v_i) \cdot (x - v_i)^n)>\val(g(v_i))$
for all $n>0$ and $i$ large and so 
$$g(x)=g(v_i)+ \sum_{n = 1}^{\deg g} (D^{(n)}g)(v_i)  (v_{i+1} - v_i)^n x_i^n\in g(v_i)(1+\mm'\cap V[x_i]),$$
which is enough. 
\hfill\ \end{proof}

\begin{Lemma} \label{kap}
Let $V \subset V'$ be an  immediate extension of valuation rings, $K\subset K'$ its fraction field extension and  $(v_i)_{i<\lambda}$ an algebraic pseudo convergent sequence in $K$, which has a pseudo limit $x$ in $V'$, algebraic over $V$, but no pseudo limit in $K$. Assume that $h=$Irr$(x,K)$ is from $V[X]$. Then 
\begin{enumerate}
\item $\val(h(v_i))<\val(h(v_j))$ for large $i<j<\lambda$,
\item if $h$ has minimal degree among the polynomials $f\in V[X]$ such that
 $\val(f(v_i))<\val(f(v_j))$ for large $i<j<\lambda$, then $V''=V'\cap K(x)$  is a filtered  union of its complete intersection $V$-subalgebras. 
\end{enumerate}
\end{Lemma}

\begin{proof} The first part follows from \cite[Corollary 5.5]{KV}. 
 Now, assume that $\deg h$ is minimal among $f$ having the property from (2). As in Lemma \ref{ka} we consider $(x_i)_{i<\lambda}$ and note that for a polynomial $g\in V[X]$ of degree $<\deg h$ we have $g(x)=du$ for some $d\in V\setminus \{0\}$ and $u\in V[x_i]$ for some $i<\lambda$ with $\val (u)=0$. An element of $V''$ has the form 
  $g(x)/t$ for some $g\in V[X]$ with $\deg g<\deg h$ and $t\in V^\setminus \{0\}$ such that $\val(g(x))\geq \val(t)$.  
 It follows that
$g(x)/t=(d/t)u\in  V[x_i]_{\mm'\cap V[x_i]}$, where $\mm'$ is the maximal ideal of $V'$. 

 Thus $V''$ is  the union of $(V[x_i]_{\mm'\cap V[x_i]})_{i< \lambda}$, which is
filtered increasing as in the proof of  \cite[Lemma 3.2]{Po}, or \cite[Lemma 15]{P}.
We have $V[x_i]\cong V[X_i]/(h_i)$, where $h_i$ is defined below.
Set $g_i=h(v_i+(v_{i+1}-v_i)X_i)\in V[X_i]$. By construction $x_i=(x-v_i)/(v_{i+1}-v_i)$ and so $g_i(x_i)=h(x)=0$. As $h$ is irreducible in $K[X]$ we get $g_i$ irreducible in $K[X_i]$ too because it is obtained from $h$ by a linear transformation. So $K[x_i]\cong K[X_i]/(g_i)$. We have $g_i=uh_i$ for some primitive polynomial $h_i$ of $V[X_i]$ and a nonzero constant $u$ of $K$. 
 Then $V''$ is a filtered  union of its  complete intersections $V$-subalgebras.   
\hfill\ \end{proof}

\begin{Remark}\label{r0}  The extension $V\subset V''$ from (2) of the above lemma is isomorphic with the one constructed in \cite[Theorem 3]{Kap}. 
\end{Remark}

We need the following elementary lemma.  

\begin{Lemma} \label{com} 
Let $B$ be a complete intersection algebra over a ring $A$ and $C$ a complete intersection  algebra over $B$. Then  $C$ is a complete intersection algebra over $A$.
\end{Lemma}

\begin{proof} Suppose that $B\cong A[X]/(f)$, where $X=(X_1,\ldots,X_n)$ and $f=(f_1,\ldots,f_r)$ is a regular sequence of elements in $A[X]$. Similarly, $C\cong B[Y]/(g)$, where $Y=(Y_1,\ldots,Y_m)$ and $g=(g_1,\ldots,g_s)$ is a regular sequence of elements in $B[Y]$. Since $A[X,Y]$ is flat over $A[X]$ we see that $f$ is a regular sequence of elements in $A[X,Y]$. Note that $g_i$ is a residue class modulo $f A[X,Y]$ of a polynomial $G_i\in A[X,Y]$,  $1\leq i\leq s$.

We claim that $f,G$ form a regular sequence of elements of $A[X,Y]$. It is enough to see that $G$ is a  regular sequence of elements of $A[X,Y]/(f)\cong B[Y]$, which is true by assumption.
So, $ C\cong A[X,Y]/(f,G)$ is a complete intersection $A$-algebra.
\hfill\ \end{proof}

We recall some results from \cite{P1} which we need in the proof of the Proposition \ref{pr}. These results are stated in \cite{P1} for the valuation rings containing a field but their proofs work easier in the mixed characteristic valuation rings.

\begin{Lemma}(\cite[Corollary 17]{P1} \label{co}
Let $V \subset V'$ be an  immediate extension of valuation rings,  $K, K'$ the fraction fields of $V, V'$, $\mm$ the maximal ideal of $V$ and  $y\in K'$ an  element  which is not in  $K$. 
Assume that  $y$ is  a pseudo limit of a pseudo convergent sequence $v=(v_j)_{1\leq j<\lambda}$ over $V$, which  has no pseudo limit in $K$.
 Set $y_j=(y-v_j)/(v_{j+1}-v_j)$.
Then for every nonzero  polynomial  $g\in V[Y]$  and every ordinal $1\leq \nu<\lambda$ one of the following statements holds.

\begin{enumerate}
\item  There exist some  $\nu<t<\lambda$ and a polynomial $g_1\in V[Y_t]$ such that 
$g(y)=g_1(y_t)$
 and    $g_1=g_1(0)+c  Y_t+g_2,$ 
for some $c\in V\setminus \{0\}$ and  $g_2\in  c\mm Y_t^2 V[Y_t]$.
\item 
There exist some  $\nu<t<\lambda$ and a polynomial $g_1\in V[Y_t]$ such that 
$y g(y)=g_1(y_t)$
 and    $g_1=g_1(0)+c  Y_t+g_2,$ 
for some $c\in V\setminus \{0\}$ and  $g_2\in  c\mm Y_t^2 V[Y_t]$.
\end{enumerate}
\end{Lemma}

\begin{Lemma}(\cite[Corollary 20]{P1} \label{co'}
Let $V \subset V'$ be an  immediate extension of valuation rings,  $K, K'$ the fraction fields of $V, V'$, $\mm$ the maximal ideal of $V$ and  $y_1,y_2\in K'$ two  elements  which are not in  $K$. 
Assume that  $y_i$, $i=1,2$ are   pseudo limits of two pseudo convergent sequences $v_i=(v_{i,j})_{1\leq j<\lambda_i}$, $i=1,2$ over $V$, which have no pseudo limits in $K$.
 Set $y_{i,j}=(y_i-v_{i,j})/(v_{i,j+1}-v_{i,j})$.
Then for every nonzero polynomial  $f\in V[Y_1,Y_2]$  and every two ordinals $\nu_i<\lambda_i$ 
there exist some  $1\leq \nu_i<t_i<\lambda_i$, $i=1,2$ and a polynomial $g\in V[Y_{1,t_1}, Y_{2,t_2}]$
 such that $g=g(0)+c_1  Y_{1,t_1}+c_2  Y_{2,t_2}+g',$ 
for some $c_1,c_2\in V$, at least one of them nonzero,   $g'\in  (c_1,c_2)\mm (Y_{1,t_1},Y_{2,t_2})^2   V[Y_{1,t_1}, Y_{2,t_2}]$ 
and $g(y_{1,t_1},y_{2,t_2})$ is one of the following elements
$f(y_1,y_2)$, \ $y_1f(y_1,y_2)$, \ $y_2f(y_1,y_2)$,  $y_1y_2f(y_1,y_2)$.
\end{Lemma}

 Next lemma is inspired by \cite[Lemma 22]{P1}.
\begin{Lemma}  \label{1}
 Let  $ V'$ be an  immediate extension  of a valuation ring $V$, $\mm, \mm'$ the maximal ideals of $V,V'$, $K\subset K'$ their fraction field extension and $y\in V'$, $y\not \in K$ an unit and algebraic element over $K$. 
Let  $f\in V[Y]$ be a nonzero polynomial and some $d\in V\setminus \{0\}$ and an unit $z\in V'$, $z\not \in K$ such that $f(y)=d z$. Then there exists 
 a complete intersection  $V$-subalgebra  of $V'$ containing $y,z$.  
\end{Lemma}
\begin{proof}
By \cite[Theorem 1]{Kap}  $y$ is a pseudo limit of a pseudo convergent sequence $v=(v_j)_{j<\lambda}$, which 
 has no pseudo limits in $K$.  Set $y_j=(y-v_j)/(v_{j+1}-v_j)$. 

 By Lemma \ref{co} there exist some  $j<\lambda$ and a polynomial $h\in V[Y_j]$ such that either 
$f(y)=h(y_j)$, or $y f(y)=h(y_j)$
 and  $h=h(0)+c  Y_j+h',$
for some   $c\in V\setminus \{0\}$ and $h'\in  c\mm Y_j^2 V[Y_j]$.   If $d|c$ then $d|h(0)$. If $f(y)=h(y_j)$ then  $z=h(y_j)/d$ is contained in  $V[y_j]_{\mm'\cap V[y_j]}$ and $y$ is too. We recall as in Lemma \ref{kap} that $V[y_j]_{\mm'\cap V[y_j]}$ is a complete intersection algebra over $V$.
 If $yf(y)=h(y_j)$  then $yz=h(y_j)/d$ and again  $  yz$ is contained in  $V[y_j]_{\mm'\cap V[y_j]}$  and so $y$ is too. Thus $z\in V[y_j]_{\mm'\cap V[Y_j]}$.

Now assume that $d$ does not divide $c$. Then $c|d$ and $c|h(0)$. If $f(y)=h(y_j)$ then consider  $g=(h-dZ)/c$ and we see that $y,z$ are contained in the etale $V[z]_{\mm'\cap V[z]}$-subalgebra $V[y_j,z]_{\mm'\cap V[y_j,z]}$ of $V'$. As in Lemma  \ref{kap} we see that $V[z]_{\mm'\cap V[z]}$ is a complete intersection over $V$. If $yf(y)=h(y_j)$ then consider $g=(h-d YZ)/c$ and we see that $y_j$ and  $yz$ are contained in the etale $V[yz]_{\mm'\cap V[yz]}$-subalgebra $V[y_j, yz]_{\mm'\cap V[y_j,yz]}$ of $V'$, which contains also $y$ and so $z$. In both cases $y,z$ are contained in a complete intersection $V$-subalgebra of $V'$ using Lemma \ref{com}. 
\hfill\ \end{proof}

\begin{Proposition} \label{pr} Let  $ V'$ be an  immediate extension  of a    valuation ring $V$,  $K\subset K'$ their fraction field extension  and $K'=K(y)$ for some $y\in V'$, $y\not\in K$ unit and  algebraic over $K$.   Then $V'$ is 
a filtered
 union of its complete intersection $V$-subalgebras.
\end{Proposition}
\begin{proof}
Let
 $f_1,\ldots, f_n\in V[Y]$ be some  polynomials such that $f_e(y)\not = 0$, $f_e(y)\not \in K$, $1\leq e\leq n$. Then there exist some $d_e\in V\setminus \{0\}$ and some units $y_e\in V'$, $y_e\not \in K$ such that $f_e(y)=d_ey_e$, $1\leq e\leq n$. We claim that there exists 
 a complete intersection $V$-subalgebra $A_{f_1,\ldots,f_n}$ of $V'$ containing $(y_e)$.  More precisely, we will choose $A_{f_1,\ldots,f_n}$ to be an etale algebra over a complete intersection $V$-algebra.
  Set $y_0=y$. By \cite[Theorem 1]{Kap}  $y_e$ is a pseudo limit of a pseudo convergent sequence $v_e=(v_{e,j_e})_{j_e<\lambda_e}$, which has no pseudo limits in $K$,  $0\leq e\leq n$. 
   Set $y_{e,j_e}=(y_e-v_{e,j_e})/(v_{e,j_e+1}-v_{e,j_e})$. 
	
	Apply induction on $n$. If $n=1$ we apply Lemma \ref{1} and we get an etale \\
	 $V[y_{0,j_0}]_{\mm'\cap V[y_{0,j_0}]}$-subalgebra  of $V'$  for some $j_0$ containing $y_0,y_1$. 
Assume that $n\geq 2$. Using induction hypothesis on $n$ we may suppose after a change of the numbering of $(y_e)$
that there exist $j_e<\lambda_e$, $0\leq e< n$ and  $h_1,\ldots, h_{n-1}\in V[Y_{0,j_0},\ldots,Y_{n-1,j_{n-1}}]$
such that $h_i((y_{e',j_{e'}}))=0$, $1\leq i< n$ and the determinant of the matrix 
$$((\partial h_i/\partial Y_{e,j_{e}}) ((y_{e',j_{e'}})))_{1\leq i < n, 0\leq e\leq n-2}$$ is not in $\mm'.$

As in the proof of Lemma \ref{1} we have to consider two cases. First assume that applying Lemma \ref{1} we arrive in the second case, namely that for some $j_0 <\lambda_0$ there exists $h_1\in V[Y_{0,j_0},Y_1]$ such that
$h_1(y_{0,j_0},y_1)=0$ and $(\partial h_1/\partial Y_{0,j_0})(y_{0,j_0},y_1)\not \in \mm'$.

 Clearly $y_{n-1}$ is algebraic over $K(y_n)$. Thus there exists a  nonzero polynomial $g_{n-1,n}\in K[Y_{n-1},Y_n]$ with $g_{n-1,n}(y_{n-1},y_n)=0$. We may choose $g_{n-1,n}\in V[Y_{n-1}, Y_n]$.

By Lemma \ref{co'}
there exist some  $j_{n-1}<t_{n-1}<\lambda_{n-1}$, $t_n<\lambda_n$ and a polynomial $h_{n-1,n}\in V[Y_{n-1,t_{n-1}}, Y_{n,t_n}]$
 such that  
$$h_{n-1,n}- h_{n-1,n}(0)=c_{n-1} Y_{n-1,t_{n-1}}+c_n  Y_{n,t_n}+h_{n-1,n}'$$
for some $c_{n-1},c_n\in V$, at least one of them nonzero, and \\
 $h_{n-1,n}' \in  (c_{n-1},c_n)\mm V[Y_{n-1,t_{n-1}}, Y_{n,t_n}]$ 
and 
$h_{n-1,n}(Y_{n-1,t_{n-1}}, Y_{n,t_n})$ corresponds to one of the following polynomials $g_{n-1,n}(Y_{n-1},Y_n),$\  \ $Y_{n-1}g_{n-1,n}(Y_{n-1},Y_n)$,\\
 $Y_n g_{n-1,n}(Y_{n-1},Y_n)$,
$Y_{n-1}Y_n g_{n-1,n}(Y_{n-1},Y_n)$.
Set $t_q=j_q$ for $0\leq q\leq n-2$.
The change from $j_e$  to $t_e$ will modify a little the coefficients of $h_i$, $1\leq i<n-1$ but their forms remain.  

Assume that  $c_{n-1}\not =0$  and $\val(c_{n-1})\leq\val(c_n)$ 
then $c_{n-1}|h_{n-1,n}(0)$ because\\
 $h_{n-1,n}(y_{n-1,t_{n-1}},y_{n,t_n})=0$ and we may change 
$h_{n-1,n}$ by $h_{n-1}=h_{n-1,n}/c_{n-1}$ independently if it was obtained from $g_{n-1,n}(Y_{n-1},Y_n)$, or $Y_{n-1}g_{n-1,n}(Y_{n-1},Y_n)$, or $Y_n g_{n-1,n}(Y_{n-1},Y_n)$, or $Y_{n-1} Y_n g_{n-1,n}(Y_{n-1},Y_n)$.  Thus a minor of maximal rank of the matrix 
$((\partial h_i/\partial Y_{e,t_e}) ((y_{e',t_{e'}})))_{1\leq i \leq n, 0\leq e< n}$ is  not contained in $\mm'$.

  We get the  etale 
   $V[y_{n,t_n}]_{\mm'\cap V[y_{n,t_n}]}$-subalgebra \\
 $A_{f_1,\ldots,f_n}= V[y_{0,t_0},\ldots, y_{n,t_n}]_{\mm'\cap V[y_{0,t_0},\ldots, y_{n,t_n}}$  of $V'$ containing $(y_{e,t_e})$ and so all $y_e$.
Certainly, we could have above also $c_n\not =0$ and $\val(c_{n-1})>\val(c_n)$
in which case change  $f_n$ with $f_{n-1}$. Hence our claim is proved.

Now assume that applying Lemma \ref{1} we arrive in the first case of the proof, namely $y_1\in V[y_{0,j_0}]_{\mm'\cap V[y_{0,j_0}]}$ for some $j_0<\lambda_0$. Then we may replace $y_0$ by
$$(v_{0,j_0+1}-v_{0,j_0})y_0+v_{0,j_0}$$
and we may omit $y_1$, that is we denote $y_e$ by $y_{e-1}$ for $1<e\leq n$. So we arrive in the case $n-1$ when we apply the induction hypothesis to show our claim.

The family  $(A_{f_1,\ldots, f_n})$ given by all finite subsets $\{f_1,\ldots,f_n\}$ of nonzero polynomials of $V[Y_0]$ is filtered and consider 
 their union $\mathcal A$. We claim that $V'={\mathcal A}$. Indeed, 
 let $f\in V[Y_0]$ and $t\in V\setminus \{0\}$  with $\val(f(y_0))>\val(t)$  (a general element of $V'$ has this form ).  We can assume  $f(y_0)=dy_1$, for some $d \in V\setminus \{0\}$ and some unit $y_1\in V'$. For some $t_e<\lambda_e$, $0\leq e\leq 1$ we find as above a complete intersection $V$-subalgebra $A_f$ of $V'$ containing   $y_{0,t_0}$, $y_{1,t_1}$ and so containing
 $(y_e)$,  $0\leq e\leq 1$.
 As  
$\val(d)\geq \val(t)$  we get  
$f(y_0)/t=(d/t)y_1\in   A_f\subset {\mathcal A}.$
\hfill\ \end{proof}

{\bf Proof of Theorem \ref{T}.}

 Firstly assume  $K'/K$ is  finite, let us say $K'=K(x_1,\ldots,x_n)$. Set $V_i=V'\cap K(x_1,\ldots,x_i)$, $1\leq i< n$ and $V_0=V$, $V_n=V'$.
Then $V_{i+1}$ is a   filtered  union of its complete intersection $V_i$-algebras for all $0\leq i<n$ by Proposition \ref{pr} and Lemma \ref{com}, which is enough. In general, express $K'$ as a filtered  union of some subfields $(K_i)_i$ of $K'$ which are finite  extensions of $K$. Then $V'$ is a filtered increasing union of $V'\cap K_i$.

\begin{Remark} If $V$ is a valuation ring of characteristic $\not =0$ then an algebraic separable valuation ring extension of it may fail to be a filtered direct limit of smooth $V$-algebras, for instance  by \cite[ Example 3.13, Theorem 6.9, Remark 6.10]{Po1} (see \cite[Sect 9, No 57]{O}); but it is a filtered union of its complete intersection $V$-subalgebras as in Theorem \ref{T}.
\end{Remark}

\end{document}